\documentclass[11pt]{amsart}

\usepackage{amsmath,amssymb,amscd,amsfonts,amsthm,verbatim}
\usepackage[all,cmtip]{xy}
\usepackage{paralist}
\usepackage[mathscr]{eucal}

\usepackage{color}
\usepackage{pdfsync}
\usepackage[]{fontenc}
\usepackage{enumerate}
\numberwithin{equation}{section}

\usepackage[pdftex]{hyperref}
\hypersetup{citecolor=blue,linktocpage}

\usepackage[colorinlistoftodos]{todonotes}

\newtheorem{thm}{Theorem}[section]
\newtheorem{lem}[thm]{Lemma}
\newtheorem{prop}[thm]{Proposition}
\newtheorem{cor}[thm]{Corollary}

\newtheorem{assu-nota}[thm]{Assumption--Notation}

\theoremstyle{definition}
\newtheorem{defn}[thm]{Definition}
\newtheorem{rem}[thm]{Remark}
\newtheorem{ex}[thm]{Example}
\newtheorem{qst}[thm]{Question}

\newcommand{\inv}{^{-1}}

\newcommand{\C}{\mathbb C}
\newcommand{\Z}{\mathbb Z}

\newcommand{\pp}{\mathbb P}

\newcommand{\OO}{\mathcal O}

\DeclareMathOperator{\Pic}{Pic}

\DeclareMathOperator{\Hom}{Hom}

\newcommand{\epsi}{\varepsilon}

\newcommand{\al}{\alpha}

\newcommand{\Ga}{\Gamma}

\newcommand{\fie}{\varphi}

\numberwithin{equation}{section}

\title{On the degree of the canonical map of a surface of general type}
\author{Margarida Mendes Lopes  and   Rita Pardini}
 \thanks {Research partially supported by  FCT/Portugal through UID/MAT/04459/2020  and by project PRIN 
 2017SSNZAW$\_$004 ``Moduli Theory and Birational Classification"  of Italian MIUR.  The first author is a member of Centro de An\'alise  Matem\'atica,  Geometria e Sistemas Din\^amicos of T\'ecnico/ Universidade de Lisboa. The second author is a member of GNSAGA of INDAM}

\dedicatory{A Ciro,   maestro e amico.}

\begin{document}

\begin{abstract} Let $X$ be a minimal complex surface of general type such that its  image $\Sigma$ via the canonical map $\fie$ is a surface; we denote by  $d$  the degree of $\fie$.
\par
In this expository work, first of all we recall the known possibilities for $\Sigma$ and $d$ when $\fie$ is not birational, which are quite a few,  and then we consider the question of producing concrete   examples for all of them.  
 We present  the two main methods  of  construction of  such examples and we give several instances of their application. We end the paper outlining the state of the art on this topic and raising several questions. 


\end{abstract}
\maketitle

\setcounter{tocdepth}{1}
\tableofcontents

\section{Introduction}

In his  epochal paper  \cite{be}  Beauville  undertook the study of the canonical map of surfaces of general type, bringing to light the great variety of possible behaviours of this map. In spite of some later refinements  (see for instance \cite{xiaodeg}, \cite{xiaopencil})  and of the great many examples in the literature (too many to  be listed here), several  questions  still remain open. 

In this paper  we focus on the case when the canonical image is a surface and the canonical map is not birational, with the aim to offer the reader a quick guide to the topic. In section 2 we summarize the known general results  and  in section 3  we describe the two main methods of construction of examples found in the literature, that is,  generating pairs and abelian covers. In section 4 we apply the methods of section 3 to construct an assortment of examples, with the aim of giving a taste of how the methods  work. 
In section 5 we  collect some  final remarks  and  several  open questions. 

We have tried to keep the exposition reader's friendly, hoping  that the paper may serve as an introduction to this fascinating topic, besides being a useful reference for surface experts. 

\medskip

\noindent {\bf Notations and Conventions:} We work over the complex numbers. All varieties are assumed to be projective and irreducible, unless otherwise specified. 
For a product variety $X_1\times X_2$  we denote by $p_i\colon X_1 \times X_2\to X_i$, $i=1,2$, the two projections and, given $L_i\in \Pic(X_i)$, we denote by $L_1\boxtimes L_2$ the line bundle $p_1^*L_1\otimes p_2^*L_2$ on $X_1\times X_2$.  We denote linear equivalence of divisors by $\equiv$. \par 
For a smooth surface $X$ we write as usual  $p_g(X):=h^0(K_X)=h^2(\OO_X)$ and $q(X):=h^0(\Omega^1_X)=h^1(\OO_X)$; in case $X$ is singular we set $p_g(X)=p_g(Y)$ and $q(X)=q(Y)$ for $Y$ any desingularization of $X$.   As usual  an effective   divisor on a surface is  said to be  {\it  normal crossings}  if its only singularities are ordinary  double points  and it is said to be  {\it  simple normal crossings}  if it is normal crossings and in addition   all its  irreducible components are  smooth.
\section{The canonical map}
In this section, we denote by $X$ a surface of general type with $p_g:=p_g(X)\ge 3$ and by $\fie\colon X \to   \pp^{p_g-1}$ the canonical map; we assume that the image $\Sigma$ of $\fie$ is a surface and we set $d:=\deg \fie$.
\subsection{The canonical image}\label{ssec: image}
The main result on the canonical image in  case it is a surface is due to Beauville and  is the following:
\begin{thm}[\cite{be}, Thm.~3.1]\label{thm: be-pg}
There are the following possibilities:
\begin{itemize}
\item[\rm(A)] $p_g(\Sigma)=0$;
\item[\rm(B)] $p_g(\Sigma)=p_g(X)$ and $\Sigma$ is a canonical surface, namely it is  the canonical image of a surface of general type $Z$ with birational canonical map. 
\end{itemize} 

\end{thm} 

As explained in \cite{be} (see Example \ref{ex2A})  for any surface $T$ with $p_g=0$  one can find minimal surfaces of general type $X$ such that  the image of the canonical   map of $X$ is  $T$,  by taking a suitable $\mathbb Z_2$-cover.  On the other hand,  surfaces in case (B)  with $d\geq 2$ were thought for a long time not to exist. The first example  of such a phenomenon is a surface with $d=2$ whose  canonical image is a quintic surface in $\pp^3$ with 20 double points (see \cite{be}, \cite{babbage}). 

\subsection{Bounds on the canonical degree}\label{ssec: bounds}
Here we look at the bounds on the degree $d$ of the canonical map. Some of the examples referred to  below can be found in  \S \ref{sec: examples}.    

Recall that  writing  $K_X\equiv M+Z$, where $Z$ is the fixed part of the canonical linear system $|K_X|$,  one has 
$$M^2\geq d \deg \Sigma.$$

Minimal surfaces $S$ of general type with birational canonical map satisfy $K_S^2\geq 3p_g(S)+ q(S)-7$  (see \cite [ Thm.~3.2]{debarre}). So 
 in case (B)  one has  $\deg \Sigma\geq 3p_g(X)+ q(\Sigma)-7$  whilst in case (A) one has $\deg \Sigma\geq p_g(X)-2$, since $\Sigma$  generates $ \pp^{p_g-1}$.

Assume now that  $X$ is   minimal. Then one has  $K_X^2\geq M^2$ and so  in case (A), 

 $$K^2_X\geq  d(p_g(X)-2)$$

 and in case (B) 
$$K^2_X\geq d(3p_g(X)+ q(\Sigma)-7).$$

 Combining the above inequalities with the Bogomolov-Miyaoka-Yau bound $K^2_X\leq 9\chi(\mathcal O_X)$ one obtains for case (A)
\begin{equation}\label{ineA}  27- 9q(X)\geq   (d-9) (p_g(X)-2)
 \end{equation}
 and for case (B)  
 \begin{equation}\label{ineB}  30-9q(X)-d q(\Sigma)\geq (d-3) (3p_g(X)-7)
 \end{equation}
 
Thus, as already pointed out  in \cite{be},  in case (A)  $d\leq 9$ if $p_g\geq 30 $ and  in case (B)  $d\leq 3$ if $p_g(X)\geq 13$.  Later on Xiao Gang in  \cite{xiaodeg} showed that in fact in case (A) if $p_g(X)>132 $   then deg $\fie\leq 8$.  For degree $8$ there exist  examples with unbounded  $p_g$ 
as shown by Beauville  (see Example \ref{ex8 Be}). More recently in \cite{bin8} Nguyen Bin produced several other such examples. 

Surfaces in case (B)  with $d\geq 2$ are much more difficult to find.  Note  that minimal surfaces in case (B) with $d=2$ must satisfy 
\begin{equation}\label{acB2}  
K^2_X\geq 6p_g(X)+2q(\Sigma)-14
 \end{equation}
whilst if $d=3$ they must satisfy
\begin{equation}\label{acB3}  
K^2_X\geq 9p_g(X)+3q(\Sigma)-21
 \end{equation} and as such are very near the Bogomolov-Miyaoka-Yau bound. Families with unbounded $p_g$  and $d=2$  are known  (see \S \ref{sec: examples}) but  for $d\geq 3$ one knows only sporadic examples.

Many other consequences of the above formulas can be drawn. To  point a few: 

\begin{itemize}

\item [(1)]  $d\leq 9$ whenever $q(X)\geq 3$; 
\item  [(2)] as noted first in \cite{Pe}, the maximum possible degree is 36 and can be attained only if $p_g=3$, $q=0$;
\item  [(3)]  if $q(X)>0$  the maximum possible degree is 27  and can be attained only if $p_g=3$, $q=1$;
\item  [(4)]  in case (B) the maximum possible degree is 9 and can be attained only if $p_g=4$, $q=0$;
\item  [(5)]  in case (B) if $d=3$, $q(X)\leq 3$.
\end{itemize}

 C. Rito  in  \cite{Rito36} uses the Borisov-Keum equations of a fake projective plane and the Borisov-Yeung equations of the Cartwright-Steger surface (the unique known minimal surface of general type with $K^2=9, p_g=q=1$)   to show the existence of a  surface with  $p_g=3, q=0$ and canonical map of degree 36, and of a surface with $p_g=3, q=1$ and canonical map of degree 27 (as in (2) and (3) above). The  first surface is an \'etale $\mathbb{Z}_2^2$-cover of the fake projective plane and the second an \'etale $\mathbb{Z}_3$-cover of the Cartwright-Steger surface.

Let us also remark that (again as explained in \cite{be}) that if $p_g\geq 25$ and  $\Sigma$ has Kodaira dimension $\geq 0$ then $d\leq 4$. This is an immediate consequence of the fact that a surface in $\pp^n$ of degree less than $2n-2$ has  Kodaira dimension $-\infty$  and if  the  degree is $2n-2$  either  it  has  Kodaira dimension $-\infty$ or it is a K3-surface that,  having $p_g=1$,   cannot be the image of a canonical map by Theorem \ref{thm: be-pg}.

\section{Two constructions}\label{sec: constructions}
We describe here the two main constructions used in the literature to produce examples of surfaces with non birational canonical map. 
\subsection{Generating pairs}\label{ssec: Gpair}
The first instance of this construction is due to Beauville (cf. \cite{Cat}, \S2.9), who used it to produce the first known unbounded sequence of surfaces falling in case (B) of Theorem \ref{thm: be-pg}.  Later the construction was studied systematically in \cite{cirifra} and \cite{cirifra2}, producing more unbounded sequences of such examples. 

The idea is to start with a finite map $h\colon V\to W$ of surfaces and a line bundle $L$ on $W$  satifying certain  conditions (cf. Definition \ref{def: Gpair}  below) and use it to construct a sequence  of surfaces  $X_n$, $n\ge 3$ such that the canonical image $\Sigma_n$ of $X_n$ is a surface and:
\begin{itemize}
\item $p_g(\Sigma_n)=0$ and the canonical map of $X_n$ has degree  twice the  degree of $h$ if the general curve $C\in |L|$ is hyperelliptic;
\item  $p_g(\Sigma_n)=p_g(X_n) $ and   the canonical map of $X_n$ has degree equal to the degree of $h$ if the general curve $C\in |L|$ is not hyperelliptic; 
\item $\lim_{n\to\infty}p_g(X_n)=+\infty$
\end{itemize}  
\medskip
 
 More precisely, we  formulate the following:
\begin{defn}\label{def: Gpair} A {\em generating pair of degree $\nu$} is  a pair $(h\colon V\to W, L)$ where  $V$ and $W$ are surfaces with at most canonical singularities, 
 $h\colon V\to W$ is  a finite  morphism of degree $\nu\ge 2$  and  $L$ is  a  line  bundle on $W$ such that   the following conditions are satisfied: 
\begin{itemize}
\item $q(W)=0$, $p_g(W)=p_g(V)$; 
\item $L^2>0$,   $h^0(L)\ge 2$;  
\item the general curve $C\in |L|$ is smooth  and irreducible of genus $g\ge 2$ and $h^*C$ is also smooth;
\item  the natural  pull back map $H^0(K_W+L)\to H^0(K_V+h^*L)$ is an isomorphism. 
\end{itemize} 
\end{defn}
\begin{rem}
In the original construction by Beauville (Example \ref{ex: Gpair1}) the surface $V$ is an   abelian surface with an irreducible principal polarization, $h\colon V\to W$ is the quotient map to the Kummer surface and $L$ is the line bundle induced on $W$ by twice the principal polarization of $V$. 
\end{rem}
\begin{rem}
The definition of generating pair  given here is slightly different from the one in \cite{cirifra}:  we do not require $V$ to be smooth nor the map $h$ to be unramified in codimension 1. On the other hand, for the sake of simplicity, we require $q(W)=0$ instead of deducing this property from the other conditions.  
Very few examples of  generating pairs  $(h\colon V\to W, L)$ are known  (cf. \cite{cirifra}, \S 3), and they fit both definitions. In addition, they all have $p_g(V)>0$. 
In view of the bounds given in \S \ref{ssec: bounds}, Corollary \ref{cor: Gpair} below  implies  that the degree $\nu$ of a generating pair is at most $3$ if $C$ is not hyperelliptic and at most 4 otherwise. In addition, in \cite[\S 7]{cirifra} a detailed analysis  is carried out, showing that also  the numerical possibilities for  $L^2$, $h^0(W,L)$ and the genus $g$  of  $C$ are very limited if $h^*C$ is not hyperelliptic. Note that if $h^*C$ is hyperelliptic, the construction below produces surfaces with canonical map of degree $2\nu$ onto a rational surface.

\end{rem}

Consider now a generating pair $(h\colon V\to W, L)$ of degree $\nu$. Let  $n\ge 3$ be an integer and consider the linear system $|M_n|:=|L\boxtimes \OO_{\pp^1}(n)|$ on $W\times \pp^1$,  let  $\Sigma_n \in |M_n|$ be a general element, let  $X_n\subset V\times \pp^1$ be the preimage of $\Sigma_n$ and let $f\colon X_n\to \Sigma_n$ the induced map. 
We have the following (cf. \cite{cirifra}): 

\begin{prop}\label{prop: Gpair} 
 In the above assumptions and notations, one has:
\begin{enumerate}
\item $\Sigma_n$ and $X_n$ are surfaces of general type with at most canonical singularities;
\item $\Sigma_n$ is minimal if $K_W+L$ is nef and $X_n$ is minimal if $K_V+h^*L$ is nef;
\item $p_g(X_n)=p_g(\Sigma_n)$, hence the canonical map of $X$ is composed with $f$;
\item the canonical image of $\Sigma_n$ is a surface;
\item the canonical map of $\Sigma_n$ is birational if $C$ is not hyperelliptic and it has degree 2 onto a rational surface otherwise. 
\end{enumerate}
\end{prop}
\begin{proof}[Sketch of proof] We just outline the proof and  refer the reader to \S 2 of \cite{cirifra} for full details.

The statement in (i) on the singularities  follows by  Bertini's theorem because of  the assumptions on $|L|$; the fact that $X_n$ and $\Sigma_n$ are of general type follows from (iv).  

By adjunction, we have  that $K_{\Sigma_n}$ is the restriction  of $(K_W+L)\boxtimes \OO_{\pp^1}(n-2)$ and $K_{X_n}$ is the restriction of $(K_V+h^*L)\boxtimes \OO_{\pp^1}(n-2)$, so (ii) follows immediately.  Standard computations and Kawamata-Viehweg's vanishing give $p_g(\Sigma_n)= (n-1)h^0(K_W+L)+ p_g(W)$ and $p_g(X_n)= (n-1) h^0(K_V+h^*L)+ p_g(V)$, so  (iii) follows immediately from  Definition \ref{def: Gpair}.

Denote by $f_n\colon \Sigma_n\to\pp^1$ the fibration induced by the second projection $W\times\pp^1\to \pp^1$: since $n\ge 3$ it is easy to see that $|K_{\Sigma_n}|$ separates the fibers of $f_n$. In addition, the general fiber $F_n$ of $f_n$ is isomorphic to  the general curve $C\in |L|$ and the restriction map $H^0((K_W+L)\boxtimes \OO_{\pp^1}(n-2))\to H^0(K_F)$ corresponds to the restriction $H^0(K_W+L)\to H^0(K_C)$. The latter map is surjective, since $q(W)=0$, hence the canonical map of $\Sigma_n$ restricts  to the canonical map of $F$ and statements  (iv) and (v) follow immediately. 
\end{proof}

The main consequence of Proposition \ref{prop: Gpair} is that, as explained at the beginning of the section,  a generating pair gives rise to an unbounded sequence of surfaces whose canonical map has a given behaviour:
\begin{cor} \label{cor: Gpair} 
 Let $(h\colon V\to W,L)$ be a generating pair of degree $\nu$ and let $C$ be the general curve of $|L|$. Then:
\begin{enumerate}
\item if $C$ is hyperelliptic, then for all $n\ge 3$ the surface $X_n$ is an example of case (A) of Theorem  \ref{thm: be-pg} with degree $d=2\nu$ and  the canonical image of $X_n$ is a rational surface; 
\item if $C$ is not hyperelliptic, then $X_n$ is an example of case (B) of Theorem  \ref{thm: be-pg} with degree $d=\nu$ and $\Sigma_n$ is the canonical image of $X_n$.
\end{enumerate} 
\end{cor}
For later reference, we record the following:
\begin{lem} \label{lem: invariants-Gpair}
Let $g$ be  the genus of $C$ and $\bar g$ the genus of $h^*C$; then the surfaces $X_n$ and $\Sigma_n$ have the following invariants:
 \begin{gather*}
 p_g(X_n) =p_g(\Sigma_n)= np_g(W) + (n-1)g, \\
q(\Sigma_n) = 0, \  q(X_n) = q(V)=\bar g-g>0,\\
  K_{\Sigma_n}^2= n(K_W^2-L^2)+8(n-1)(g-1)\\
 K_{X_n}^2 =  n(K_V^2-\nu L^2)+8(n-1)(\bar g -1)
\end{gather*}
\end{lem}
\begin{proof}
Definition \ref{def: Gpair}  implies that $L$ is nef and big, so  $L\boxtimes \OO_{\pp^1}(n)$ is also nef and big, and  using Kawamata-Viehweg's vanishing one obtains easily $q(\Sigma_n)=q(W)=0$,  $q(X_n)=q(V)$,  $h^0(K_W+L)=\chi(K_W+L)$ and $h^0(K_V+h^*L)=\chi(K_V+h^*L)$. Riemann-Roch and the condition $h^0(K_W+L)=h^0(K_V+h^*L)$ in Definition \ref{def: Gpair} now give $q(V)=\bar g-g$.  Then the remaining formulae are  easily computed, recalling from  the proof of Proposition \ref{prop: Gpair} that   $p_g(\Sigma_n)=(n-1)h^0(K_W+L)+ p_g(W)$ and applying the adjunction formula to compute $K^2_{\Sigma_n}$ and $K^2_{X_n}$.
\end{proof}
\begin{rem} 
Lemma \ref{lem: invariants-Gpair}  shows that all the examples of  surfaces with non birational canonical map arising from a generating pair are irregular. However, in  \cite{cirifra2} unbounded sequences $Y_n$  of  surfaces as in case (B) of Theorem \ref{thm: be-pg} with  $d=2$ and $q(Y_n)=0$ are constructed: one starts with a carefully chosen sequence $X_n$ of surfaces arising from a generating pair and admitting a $\Z_3$-action, and sets $Y_n:=X_n/\Z_3$. The action of $\Z_3$ kills the irregularity but preserves the features of the canonical map. 
\end{rem}
\subsection{Abelian covers}\label{ssec: AbCovers}
The standard reference for abelian covers is \cite{ritaabel}; in this section we summarize the properties we need and analyze in detail the base locus of the canonical system of a cover.

Let $G$ be a finite abelian group; we denote by $G^*$ the group $\Hom(G,\C^*)$ of characters of $G$.
 Here by a  {\em $G$-cover} we mean  a finite map  $f\colon X\to Y$ that is the quotient map for a faithful $G$-action on $X$, with  $Y$ irreducible and $X$  connected. If we assume $X$ normal and $Y$ smooth, then $f$ is automatically flat and  the $G$-action induces a decomposition 
\begin{equation}\label{dec}
 f_*\OO_X=\OO_Y\oplus\ \bigoplus_{\chi\ne 1}L_{\chi}\inv,
\end{equation}
where the  $L_{\chi}$ are line bundles. Note that $L_{\chi}$ is non-trivial for $\chi\ne1$ by the assumption that $X$ be connected.
Since $f$  is flat,  the  (reduced) branch locus   of $f$ is a divisor $D$. For a component $\Delta$  of $D$ we define  the {\em inertia subgroup} of $\Delta$ as  the subgroup  $H<G$ consisting of the elements fixing $f\inv(\Delta)$ pointwise. The group $H$ is cyclic and the natural representation of $H$ on the normal bundle to $f\inv(\Delta)$ at a general point defines a generator  $\psi$ of $H^*$.  
So we can group together the components of $D$ with  same inertia group and character and write  $D=\sum_{(H,\psi)}D_{(H,\psi)}$, where $(H,\psi)$ runs over all the pairs (cyclic subgroup $H<G$, generator of $H^*$). The line bundles $L_{\chi}$ and the effective divisors $D_{(H,\psi)}$ are the {\em building data} of the cover and satisfy the following {\em fundamental relations}: 
\begin{equation}\label{eq: fund-rel}
L_{\chi}+L_{\chi'}\equiv L_{\chi\chi'}+\sum \epsi_{\chi,\chi'}^{(H,\psi)}D_{(H,\psi)}, \quad \forall
\chi,\chi'\in G^*,
\end{equation}
where the coefficient $\epsi_{\chi,\chi'}^{(H,\psi)}$ is defined as follows.
Denote by $m_H$ the order of $H$. Let $r_{\chi}^{(H,\psi)}$ be the smallest nonnegative integer
such that 
$$\chi|_H=\psi^{r_{\chi}^{(H,\psi)}}.$$ Then we have  $\epsi_{\chi,\chi'}^{(H,\psi)}=1$ if 
$r_{\chi}^{(H,\psi)}+r_{\chi'}^{(H,\psi)}\ge m_H$ and $\epsi_{\chi,\chi'}^{(H,\psi)}=0$ otherwise.

The building data encode all the information on a $G$-cover, as explained in the following structure theorem:
\begin{thm}[\cite{ritaabel}, Thm. 2.1]\label{thm: structure} (Notation as above).\newline
Let $Y$ be a smooth irreducible projective variety, let  $\{L_{\chi}\}_{\chi\ne 1}$ be line bundles of $Y$ such that $L\chi\ne
\OO_Y$ for every $\chi$
 and
let $D_{(H,\psi)}$ be effective divisors such that $D:=\sum_{(H,\psi)} D_{(H,\psi)}$ is reduced. Then:
 \begin{enumerate}
\item  $\{L_{\chi}$, $D_{(H,\psi)}\}$ are the building data  of an 
abelian cover $f\colon X\to Y$ with $X$ normal if and only if they satisfy the fundamental relations \eqref{eq: fund-rel};

\item the building data determine $f\colon X\to Y$ up to  $G-$equivariant isomorphism.
\end{enumerate}
\end{thm}

Theorem \ref{thm: structure} is satisfactory from the  theoretical point of view, however it is not so  easy to
apply, since the number of equations in \ref{eq: fund-rel} grows quadratically with the order of $G$.
In fact, it is possible to reduce the number of equations to the rank of the group $G$; in particular, cyclic covers are described by one relation. The price that
one has to pay  for this is that the formulation of the theorem is not  intrinsic any more. 
So, choose  characters  $\chi_1,\dots \chi_k\in G^*$  such that $G^*=<\chi_1>\oplus\dots\oplus 
<\chi_k>$. Denote by   $d_i$ the order of $\chi_i$, write  $L_i:=L_{\chi_i}$, $i=1, \dots k$ and
$r^{(H,\psi)}_i:=r^{(H,\psi)}_{\chi_i}$. We call  $L_i$, $D_{(H,\psi)}$  the {\em reduced building
data} of the cover and we have the following {\em reduced fundamental relations}:
\begin{equation}\label{eq: redfund}
 d_iL_i\equiv \sum _{(H,\psi)}\frac{d_i r_i^{(H,\psi)}}{m_H}D_{(H,\psi)},\quad i=1,\dots k
\end{equation}
We have the following ``reduced version'' of Theorem \ref{thm: structure}:
\begin{thm}[\cite{ritaabel}, Prop. 2.1]\label{prop: redstructure} (Notation as above).\newline
Let $Y$ be a smooth projective variety, let  $\{L_i\}_{i=1,\dots k}$ be line bundles of $Y$ and
let $D_{(H,\psi)}$ be effective divisors such that $D:=\sum D_{(H,\psi)}$ is reduced. 

Then:
 \begin{enumerate}
\item  $\{L_i, D_{(H,\psi)} \}$  can be extended to a set of building data $\{L_\chi, D_{(H,\psi)}\}$  satisfying
\eqref{eq: fund-rel} if and only if  $\{L_i, D_{(H,\psi)} \}$ satisfy \eqref{eq: redfund};
\item  $\{L_i ,D_{(H,\psi)}\}$ uniquely determine the $L_{\chi}$.  
\end{enumerate}
\end{thm}
\begin{rem}\label{rem: reduced}  Let  $f\colon X\to Y$ be a $G$-cover with reduced building data $\{L_i, D_{(H,\psi)} \}$. Given a character $\chi\in G^*$, it is possible to  compute explicitly $L_{\chi}$ from the reduced building data as follows. Write $\chi=\chi_1^{\alpha_1}\cdots \chi_k^{\alpha_k}$ with $0\le \alpha_k<d_i$, $i=1,\dots k$ and for every pair $(H, \psi)$ set $\beta_{(H,\psi)}:=\big \lfloor\frac{\sum \alpha_i r_i^{(H,\psi)}}{m_H}\big\rfloor$; then:
\begin{equation}
L_\chi\equiv \sum_{i=1}^k \alpha_iL_i-\sum_{(H,\psi)} \beta_{(H,\psi)} D_{(H,\psi)}.  
\end{equation}
If $\chi$ has order $d$ and $\Pic(Y)$ has no $d$-torsion (e.g., $Y$ is simply connected), then $L_\chi$ can also be computed as the only solution in $\Pic(Y)$ of the equation  
\begin{equation}\label{eq: d-power}
dL_{\chi}\equiv \sum_{(H,\psi)}\frac {dr_{\chi}^{(H,\psi)}}{m_H} D_{(H,\psi)}.
\end{equation}
\end{rem}

\begin{rem}[Cyclic covers]\label{rem: Zd-covers} 
 Let $G=\Z_d$,  let $\chi\in G^*$ be a generator and let $L=L_{\chi}$. By Theorem \ref{prop: redstructure} in this case the reduced building data have to  satisfy just one relation. 
Choose a primitive $d$-th root $\zeta$ of $1$ and let $g\in G$ be the generator such that $\chi(g)=\zeta$. Given a pair $(H,\psi)$,   there is a unique integer  $0 <\al<d$ such that $g^{\al}$ generates $H$ and $\psi(g^{\al})=\zeta^{d/m_H}$. Conversely, $\al$ determines the pair $(H,\psi)$ completely.  
Hence we can decompose $D=\sum_ {\al}D_{\al}$.  

Here there is only one  reduced fundamental relation:
\begin{equation}
dL\equiv\sum_{\al}\al D_{\al} 
\end{equation}
If $D=D_{\al}$ for some $\al$ such that $(\al, d)=1$, then it is possible to choose $\chi$ in such a way that the reduced fundamental relation is: $$dL\equiv D.$$
This is called a  {\em simple cyclic cover}. Simple cyclic covers are especially easy to describe: they can be realized as hypersurfaces in the total space $V(L)$ of the line bundle $L$ and the adjunction formula \eqref{eq: adjunction} below simplifies to:
\begin{equation}\label{eq: simple-adj}
K_X\equiv f^*(K_Y+(d-1)L)
\end{equation}
\end{rem}

\begin{rem}[$\Z_p^k$-covers]\label{rem: Zp-covers}
Take $G=\Z_p^k$, with $p$ a prime number. If we fix a $p$-th root $\zeta\ne 1$ of 1 we can take the ``logarithm'' of any  character $\chi$ of $G$ and regard $G^*$ as a the dual vector space of $G$. In addition, there is a bijection  of $G\setminus \{0\}$ with the set of pairs (character, subgroup) that associates  to $v\in G\setminus \{0\}$ the subgroup $H$ generated by $v$ and the character $\psi\in H^*$ such that $\psi(v)=1$ (recall that we are viewing characters as  $\Z_p$-linear functionals).  If we denote by $\overline m$ the smallest non-negative representative of a class $m\in \Z_p$, the fundamental relations \eqref{eq: fund-rel} read 
\begin{equation}\label{eq: fund-rel-Zp}
L_{\chi}+L_{\chi'}\equiv L_{\chi+\chi'}+\sum_{v\ne 1}\lfloor\frac{\overline{\chi(v)}+\overline{\chi'(v)}}{p}\rfloor D_v, \quad \forall
\chi,\chi'\in G^*,
\end{equation}
and the reduced fundamental relations \eqref{eq: redfund} read: 
\begin{equation}\label{eq: redfund-Zp}
 pL_i\equiv \sum _{v\ne 1}\overline{\chi_i(v)}D_v,\quad i=1,\dots k.
 \end{equation}
 Similarly, \eqref{eq: d-power} becomes 
\begin{equation}\label{eq: p-power}
pL_{\chi}\equiv \sum _{v\ne 1}\overline{\chi(v)}D_v.
\end{equation}
\end{rem}
\bigskip

 As one would expect in view of Theorem \ref{thm: structure}, geometrical properties of a $G$-cover can be read off the building data. For instance the cover is {\em totally ramified}, namely it does not factor through a non trivial \'etale cover of $Y$, iff  the inertia subgroups of the components of $D$ generate $G$.  Here is the criterion for smoothness:
 \begin{prop}[\cite{ritaabel}, Prop. 3.1]\label{prop: smoothcover} 
Let $f\colon X\to Y$ be an abelian cover  with branch divisor $D=\sum_{(H,\psi)} D_{(H,\psi)}$,   let $y\in Y$ be a point,  let $D_1,\dots
D_s$ be the irreducible components of $D$ that contain $y$ and let $(H_i,\psi_i)$ be the pair
subgroup-character associated to $D_i$, $i=1,\dots s$. Then $X$ is smooth above $y$ if and only if:
\begin{enumerate}
\item[(1)] $D_i$ is smooth at $y$ for every $i$;

\item[(2)]  $D$ is  normal crossings  at $y$;

\item[(3)] the natural map $H_1\oplus\dots\oplus  H_s\to G$ is injective.  
\end{enumerate}
\end{prop}

One can  compute the numerical invariants of $X$ in terms of the corresponding invariants of $Y$
and of the building data. Let  $d$ denote the exponent of $G$, namely the least common multiple of the orders of the elements of $G$.  Then $dK_X$ is Cartier and one has:
\begin{equation}\label{eq: adjunction}
dK_X\equiv f^*(dK_Y+\sum \frac{d(m_H-1)}{m_H}D_{(H,\psi)})
\end{equation}
As a consequence, it makes sense to compute:
\begin{equation}\label{eq: K2}
 K_X^n=|G|\left(K_Y+\sum_{(H,\psi)} \frac{(m_H-1)}{m_H}D_{(H,\psi)}\right)^n.
\end{equation}
In addition one has:
\begin{gather}\label{eq: hi}
h^i(\OO_X)=h^i(\OO_Y)+\sum_{\chi\ne 1}h^i(L_{\chi}\inv),\\ \chi(\OO_X)=\chi(\OO_Y)+\sum_{\chi\ne 1}\chi(L_{\chi}\inv)\nonumber
\end{gather}
The group $G$ acts also on the canonical sheaf $\omega_X=\OO_X(K_X)$ and  the   sheaf $f_*\omega_X$ decomposes under the action of $G$ as follows:

\begin{equation}\label{eq: omega}
f_*\omega_X= \omega_Y\oplus \ \bigoplus_{\chi\ne 1}\omega_Y\otimes L_{\chi}
\end{equation} 
where $G$ acts trivially on  $ \omega_Y$ and acts  on  $\omega_Y\otimes L_{\chi}$ via the character $\chi\inv$. 
\begin{rem}\label{rem: omega}
Equation \eqref{eq: omega} gives  the following $G$-equivariant decomposition:
\begin{equation}\label{pg}
H^0(X,K_X)=H^0(Y,K_Y)\oplus\ \bigoplus_{\chi\ne 1}H^0(Y,K_Y+L_{\chi\inv}),
\end{equation}
where $G$ acts on $H^0(Y,K_Y+L_{\chi\inv})$ via the character $\chi$. So if we have a  $G$-cover of smooth varieties $f\colon X\to Y$  such that $h^0(K_Y+L_{\chi})\ne 0$ for precisely one character $\chi$, then the canonical system $|K_X|$ is pulled back from $Y$ and the canonical map of $X$ factorizes though $f$. If in addition the system $|K_Y+L_{\chi}|$ is birational, then $f$ is birationally equivalent to the canonical map of $X$. This observation can be used to construct examples of case (A) of Theorem \ref{thm: be-pg} (see \S \ref{ssec: ab-ex}).
\smallskip

Assume now instead that  $\Gamma$ is a subgroup of $G$ such that $h^0(K_Y+L_{\chi})=0$ for every $\chi\in\Gamma^{\perp}$. Then the canonical map of $X$ factorizes via the quotient map $X\to Z:=X/\Ga$. The surface $Z$ thus constructed is  in general singular, but it has  rational singularities and $p_g(Z)=p_g(X)$. So the canonical map of $X$ is the composition of $X\to Z$ with the canonical map of (a smooth model of) $Z$; in particular,   if the canonical map of $Z$ is birational then $X\to Z$ is essentially the canonical map of $X$. This observation can be used to construct examples of case (B) of Theorem \ref{thm: be-pg} (see  \S \ref{ssec: ab-ex}).
The difficulty in applying this method is that one  needs $L_{\chi}$ to be ``very small'' for $\chi\notin \Ga^{\perp}$, in order to satisfy condition \ref{pg}, and at the same time the building data of the cover must be sufficiently positive for $X$ to be a surface of general type. 
\end{rem}

We close this section with a result that is useful in determining  the  base locus of the canonical system of a $G$-cover:
\begin{prop} \label{prop: base-locus}
Let $f\colon X\to Y$ be a $G$-cover of smooth varieties with building data $\{L_\chi, D_{(H,\psi)}\}$ and write $f^*D_{(H,\psi)}=m_HR_{(H,\psi)}$.
Then the zero locus of the map $f^*(\omega_Y\otimes L_{\chi})\to \omega_X$ induced by the decomposition \eqref{eq: omega} is equal to $$\sum_{(H,\psi)}(m_H-1 -r_{\chi}^{(H,\psi)})R_{(H,\psi)}.$$
\end{prop}
\begin{proof} Denote by $d$ the exponent of $G$; by \eqref{eq: adjunction}, we have an isomorphism $$\omega_X^{\otimes d}\cong f^*(\omega_Y^{\otimes d}(\sum_{(H,\psi)}\frac{d}{m_H}(m_H-1)D_{(H,\psi)}),$$ 
(the $d$-canonical forms on $Y$  with poles on $\sum_{(H,\psi)}\frac{d}{m_H}(m_H-1)D_{(H,\psi)}$ pull back to regular $d$-canonical  forms on $X$).
In addition,  by \eqref{eq: d-power} we have an isomorphism
 $$(\omega_Y\otimes L_{\chi})^{\otimes d} \cong \omega_Y^{\otimes d}\left(\sum_{(H,\psi)}\frac{d}{m_H}r_{\chi}^{(H,\psi)}D_{(H,\psi)}\right)$$
 The map $f^*(\omega_Y\otimes L_{\chi})^{\otimes d}\to \omega_X^{\otimes d}$ given  by taking the $d$-th tensor power of $f^*(\omega_Y\otimes L_{\chi})\to \omega_X$ is induced via $f^*$ by the natural inclusion  $$ \omega_Y^{\otimes d}(\sum_{(H,\psi)}\frac{d}{m_H}r_{\chi}^{(H,\psi)}D_{(H,\psi)})\longrightarrow \omega_Y^{\otimes d}(\sum_{(H,\psi)}\frac{d}{m_H}(m_H-1)D_{(H,\psi)})$$
 and so it vanishes along
$$f^*\left(\sum_{(H,\psi)}\frac{d}{m_H}(m_H-1-r_{\chi}^{(H,\psi)})D_{(H,\psi)}\right)= d\left(\sum_{(H,\psi)}(m_H-1-r_{\chi}^{(H,\psi)})R_{(H,\psi)}\right).$$ Since the above divisor is equal to $d$ times the zero divisor of $f^*(\omega_Y\otimes L_{\chi})$, the claim is proven.
\end{proof} 
\begin{rem} By  Proposition \ref{prop: base-locus}, $|K_X|$ is generated by  all the effective divisors  $ |f^*(\omega_Y\otimes L_{\chi})|+ \sum_{(H,\psi)}(m_H-1 -r_{\chi}^{(H,\psi)})R_{(H,\psi)}$ such that $h^0(\omega_Y\otimes L_{\chi})\neq 0$. So we obtain an effective way of understanding the base locus of the canonical map of $X$ (for an application see  Examples \ref{ex: pg} and \ref{ex6}). 
\end{rem}

\section{Examples}\label{sec: examples}
In this section we present a sampling of examples of the various possibilities for the degree and the canonical image of a surface of general type. We do not aim at completeness, but  rather  at giving  the reader a feeling for  how the methods of  \S \ref{ssec: Gpair} and \S \ref{ssec: AbCovers} work.  For this reason  we have chosen to list separately the examples obtained via  generating pairs and via abelian covers. 

\subsection{Examples arising from generating pairs}\label{ssec: Gpair-ex}
Here we list examples obtained as in \S \ref{ssec: Gpair}, using freely the notation introduced there. 
 We refer the reader to \cite[\S 3]{cirifra} for a detailed description of the various  generating pairs. 

\begin{ex}[Case (B), $d=2$ - Beauville's example] \label{ex: Gpair1}
We take $V$ an abelian surface with an irreducible principal polarization, $h\colon V\to W$ the quotient map to the Kummer surface and $L$ the line bundle on $W$ induced by twice the principal polarization. This gives a generating pair with $L^2=4$, $\nu=2$ and $C$  a non-hyperelliptic curve of genus 3. 

By  Proposition \ref{prop: Gpair}, the construction produces a sequence of minimal surfaces $X_n$ whose canonical map is $2$-to-$1$ onto a canonically embedded  minimal surface $\Sigma_n$. By Lemma \ref{lem: invariants-Gpair}   we have:
$$K^2_{\Sigma_n}=3p_g(\Sigma_n)-7;\quad K^2_{X_n} =6p_g(X_n)-14,\quad q(X_n)=2.$$
As seen in inequality  \eqref{acB2} of \S \ref{ssec: bounds}, a surface as in case (B) of Theorem \ref{thm: be-pg} satisfies $K^2_{X} \ge 6p_g(X)-14$, so this is a limit case. In \cite{nagoya} it is proven that these are essentially the only surfaces as in case (B) of Theorem \ref{thm: be-pg} with $K^2=6p_g-14$ and $q\ge 2$.

\end{ex}
\begin{ex}[Case (B), $d=2$] \label{ex: Gpair2}
We consider a generating pair constructed as follows. Take  $A$ an abelian  surface with an irreducible principal polarization,  let $M$ be a symmetric theta divisor and let $V\to A$  be the double cover given by the relation $2M\equiv B$, where $B\in |2M|$ is general.  The surface $V$ is smooth of general type with $K^2_V=4$, $p_g(V)=q(V)=2$;  using the fact that all curves in $|2M|$ are symmetric one shows the existence of an involution $
\iota$ of $V$ that lifts multiplication by $-1$ in $A$. The involution $\iota$ has 20 isolated fixed points and acts trivially on $|K_V|$. We let $h\colon V\to W:=V/\iota$ be the quotient map:  the surface $W$ has canonical singularities and  is minimal of general type with $K^2_W=2$, $p_g(W)=2$ and $q(W)=0$. So we have $h^0(2K_W)=h^0(2K_V)=5$ and setting $L:=K_W$ one obtains a generating pair. Using the fact that $W$ is by construction a double cover of the Kummer surface of $A$ one can show that the general $C\in |L|$ is non hyperelliptic. 

By  Proposition \ref{prop: Gpair}, the construction produces a sequence of minimal surfaces $X_n$ whose canonical map is $2$-to-$1$ onto a canonically embedded  minimal surface $\Sigma_n$. By Lemma \ref{lem: invariants-Gpair}   we have:
$$5K^2_{\Sigma_n}=16p_g(\Sigma_n)-32;\quad 5K^2_{X_n} =32p_g(X_n)-64,\quad q(X_n)=2.$$
\end{ex}

\begin{ex}[Case (B), $d=2$] \label{ex: Gpair3}
This example is similar to the previous one, in that the map $h\colon V\to W$ is the bicanonical map of $V$ and $L=K_W$. Here we consider a non-hyperelliptic curve $\Gamma$ of genus $3$ and take as   $V$ the symmetric product $S^2\Gamma$. The surface $V$ is smooth minimal of general ype with $K^2_V=6$ and $p_g(V)=q(V)=3$; the bicanonical map  $h\colon V\to W$ is the quotient map for the involution $\iota$ of $V$ that sends $p+q\in V$ to the only effective divisor linearly equivalent to $K_{\Gamma}-p-q$. The surface $W$ is minimal, has canonical singularities and invariants $K^2_W=3$,  $p_g(W)=3$, $q(W)=0$. 

By  Proposition \ref{prop: Gpair}, the construction produces a sequence of minimal surfaces $X_n$ whose canonical map is $2$-to-$1$ onto a canonically embedded  minimal surface $\Sigma_n$. By Lemma \ref{lem: invariants-Gpair}   we have:
$$7K^2_{\Sigma_n}=24p_g(\Sigma_n)-72;\quad 7K^2_{X_n} =48p_g(X_n)-144,\quad q(X_n)=3.$$

\end{ex}
\begin{ex}[Case (A), $d=6$]\label{ex: GpairA} For a detailed description   of this generating pair  see  \cite{BL}.
Let $B\subset \pp^2$ be a sextic curve obtained as the dual curve of a smooth plane cubic and let $W\to \pp^2$ be the double cover branched on $B$. The singularities of $B$ are 9 cusps, that give 9 singularities of type $A_2$ on $W$, so $W$ is a K3 surface with canonical singularities. We denote by $L$ the pull back of $\OO_{\pp^2}(1)$ on $W$, so we have $L^2=2$ and the general curve $C$ of $|L|$ is smooth of genus 2. There exists a $\Z_3$-cover $h\colon V\to W$ branched precisely over the singularities of $V$. The surface $V$ is abelian and $h^*L$ is a polarization of type $(1,3)$, so $h^0(K_V+h^*L)=h^0(h^*L)=3=h^0(L)=h^0(K_W+L)$ and we have a generating pair with $C$ hyperelliptic. 

By  Proposition \ref{prop: Gpair}, the construction produces a sequence of minimal surfaces $X_n$ whose canonical map is $6$-to-$1$ onto a rational surface  $\Sigma_n$. By Lemma \ref{lem: invariants-Gpair}   we have:
$$K^2_{\Sigma_n}=2p_g(\Sigma_n)-4;\quad K^2_{X_n} =6p_g(X_n)-12,\quad q(X_n)=2.$$

\end{ex}
\subsection{Examples constructed as abelian covers} \label{ssec: ab-ex}
While all the constructions of Section \ref{ssec: Gpair-ex} give sequences of examples with unbounded invariants, some of the constructions in this section give sporadic examples.

In all these examples we have $G=\Z_p^k$ with $p$ a prime and we use the additive notation for characters as explained in Remark \ref{rem: Zp-covers}; the  group  of characters is identified with $\Z_p^k$ via the dual basis of the canonical basis of $G=\Z_p^k$.

\begin{ex}[case (A), $p_g=3$, $d=3,\dots 9$]\label{ex: pg}
We take $G=\Z_2^2$  and  $Y$  a del Pezzo surface of degree $d\ge 3$. For every  $0\ne v\in G$ we choose a  curve  $D_v$  in $|-K_Y|$ in such a way that $D:= \sum_vD_v$ is a simple normal crossings divisor. The reduced  fundamental relations are 
$2L_{10}\equiv 2L_{01}\equiv  -2K_Y$, so that $L_{10}\equiv L_{01}\equiv -K_Y$ is the only solution. By Remark \ref{rem: reduced} we get $L_{11}\equiv -K_Y$, too.
The corresponding cover $f\colon X\to Y$ is smooth by Proposition \ref{prop: smoothcover} and using equations \eqref{eq: hi} and \eqref{eq: K2} we get $K^2_X=d$, $q(X)=0$ and $p_g(X)=3$. 
By \eqref{eq: adjunction} we have that $2K_X\equiv f^*(-K_Y)$ is ample and thus $X$ is minimal of general type. 
Finally, by  Proposition \ref{prop: base-locus} we see that the system $|K_X|$ is spanned by the three  curves $R_v:=f\inv D_v$, $0\ne v\in G$. So $|K_X|$ is free and maps $X$ $d$-to-1 to $\pp^2$.
\end{ex}
\begin{ex}[Case (A), $p_g=3$, $d=16$]\label{ex: 16}
This example is due to Persson, \cite{Pe}.    Let $Y=\pp^2$ and denote by $h$ the class of a line. Take $G=\Z_2^4$,  let $\chi_0\in G^*$ be the character  that maps $v\in G$ to the sum of its coordinates and consider the following building data:
\begin{itemize}
 \item[--]  $D_v$ is a line  if  $\chi_0(v)=1$, and  $D_v=0$  otherwise;
 \item[--] $L_{\chi_0}\equiv 4h$ and $L_{\chi}\equiv 2h$ if $\chi\ne \chi_0$.
\end{itemize}
We also assume that the lines $D_v$ are in general position. 
The surface $X$ is smooth by Proposition \ref{prop: smoothcover}; by \eqref{eq: adjunction} we have  $2K_X\equiv f^*(2h)$ and therefore $X$ is minimal of general type with $K^2_X=16$. Using \eqref{eq: hi} one computes $q(X)=0$ and $p_g(X)=3$.
One has $h^0(K_Y+L_{\chi})=0$ for every $\chi\ne \chi_0$, so the canonical map of $X$ coincides with $f$ (see  Remark \ref{rem: omega}). \end{ex}
\begin{ex}[case (A), $d=2$, unbounded $p_g$] \label{ex2A}
Any surface $Y$ with $p_g(Y)=0$ occurs as the canonical image of a minimal  surface $X$. This fact was noted for the first time in \cite{be}. 

 Choose a line bundle $L$ of $Y$ such that $|2L|$ is base point free and $K_Y+L$ is very ample and let $f\colon X \to Y$ be the double cover given by the relation $2L\equiv D$, where $D$ is a general element of $|2L|$. The surface $X$ is smooth by Proposition \ref{prop: smoothcover} and $K_X\equiv f^*(K_Y+L)$ is ample, so $X$ is minimal of general type. 
In addition we have $|K_X|=f^*|K_Y+L|$ (cf. Remark \ref{rem: omega})  and therefore the canonical map of $Y$ is just $f$ followed by the embedding defined by $|K_Y+L|$.
\end{ex}
\begin{ex}[Case (A), $d=8$, unbounded $p_g$]  \label{ex8 Be} 
This example is due to  Beauville  (\cite[ Exemple 4.3]{be}) and the construction is similar to the previous one. 
Take $Y=C\times \pp^1$, where $C$ is a non-hyperelliptic curve of genus 3. 
Let $\eta \in \Pic(C)$ be a non-zero 2-torsion element and set $L=\eta\boxtimes \OO_{\pp^1}(m)$, $m\ge 3$; pick a smooth divisor $D\in |2L|=|2mF|$, where $F$  is a fiber of the projection $Y\to \pp^1$. We denote by $f\colon X\to Y$ the double cover given by the relation $2L\equiv D$. As in the previous example, one checks that $X$ is a smooth minimal surface of general type and that the canonical map of $X$ is the composition of $f$ with the map given by the system $|K_Y+L|=|(K_C+\eta)\boxtimes \OO_{\pp^1}(m-2)|$. 
Since $C$ is not hyperelliptic, the system $|K_C+\eta|$ is free of degree 4. So $|(K_C+\eta)\boxtimes \OO_{\pp^1}(m-2)|$ gives a degree 4 map to $\pp^1\times \pp^1$, and the canonical map of $X$ has degree 8. 

The invariants of $X$ are: $K^2_X=16(m-2)$, $p_g(X)=2m-2$, $q(X)=3$.
\end{ex}

\begin{ex}[Case (A), $d=6$, unbounded $p_g$] \label{ex6} 
We take $G=\Z_3^2$ and $Y=\pp^1\times \pp^1$; we denote by $|F_1|$ and $|F_2|$  the two rulings of $\pp^1\times \pp^1$. 

We consider the $G$-cover $f\colon X\to Y$ given by the following choice of building data:
\begin{itemize}
 \item[--]  $D_{10},D_{20},D_{01},D_{02}$ are  distinct  fibers  of  $|F_1|$;
 \item[--] $D_{11}$ and $D_{22}$  are distinct fibres of  $|F_2|$; 
\item[--]   $D_{12}$ and $D_{21}$ are general elements of $|mF_2|$, $m\ge 3$; 
 \item[--] $L_{10}\equiv L_{01}\equiv L_{20}\equiv L_{02}\equiv F_1+(m+1)F_2$;
 \item[--] $L_{11}\equiv L_{22}\equiv 2F_1+F_2$;
 \item[--] $L_{12}\equiv L_{21}\equiv 2F_1+mF_2$.
\end{itemize}
The surface $X$ is smooth by Proposition \ref{prop: smoothcover}; by \eqref{eq: adjunction} we have  $3K_X\equiv f^*(2F_1+(4m-2)F_2)$ and therefore $X$ is minimal of general type with $K^2_X=16m-8$. Using \eqref{eq: hi} one computes $q(X)=0$ and $p_g(X)=2m-2$.

Denote by $\Gamma<G$ the cyclic subgroup generated by $(1,1)$; one has $h^0(K_Y+L_{\chi})=0$ for every $\chi\notin \Gamma^{\perp}$, so the canonical map of $X$ is the composition of the  quotient map $X\to Z:=X/\Gamma$ with the canonical map of a smooth model of $Z$ (see  Remark \ref{rem: omega}). 
The surface $Z$  is a   $\mathbb Z_3$-cover of  $Y$  with building data 
$D_1=D_{10}+D_{02}+D_{21}$ and $D_2=D_{20}+D_{01}+D_{12}$, $L_1=L_{12}$, $L_2=L_{21}$. The singularities of $Z$ are either of type  $A_2$, occurring over the singular points of $D_1$ and $D_2$, or of type  $\frac 13(1,1)$, occurring over the intersection points of $D_1$ and $D_2$; the minimal desingularization  $\tilde Z$ satisfies 
$K_{\tilde Z}^2= 4m-8$, $q(\tilde Z)=0$, $p_g(\tilde Z)=2m-2$. In addition, it is not hard to check that $K_{\tilde Z}$ is nef, hence  $\tilde Z$ is minimal with $K^2=2p_g-4$, namely it is a Horikawa surface. By  \cite[Lemma 1.1]{ho1}  the canonical map of $\tilde Z$ is $2$-to-1 onto a minimal ruled surface of $\pp^{2m-3}$ (note that the  fibration $|F_2|$ of $Y$ induces a genus $2$ fibration on $\tilde Z$).  So the canonical map of $X$ has degree 6. 

For $v\in G\setminus \{0\}$ we write as usual $f^*D_v=3R_v$. By   Proposition \ref{prop: base-locus} the canonical system $|K_X|$ is generated by the following linear subsystems:
$$f^*|K_Y+L_{12}|+ R_{10}+ R_{02}+R_{21}+2R_{11}+2R_{22}$$ and  $$f^*|K_Y+L_{21}|+R_{20}+ R_{01}+ R_{12}+ 2R_{11}+ 2R_{22}.$$ Since the systems $|K_Y+L_{12}|$ and $|K_Y+L_{21}|$ are base point free, it follows that   the fixed locus of $|K_X|$ consists of the divisor  $2R_{11}+2R_{22}$ and of $4m$ simple base points which are the inverse images of  the intersection points of  $D_1$ and $D_2$. 
\end{ex}

\begin{ex}[Case (B), $d=3$]  \label{ex3 tan} 
This example is due to Tan (\cite{tan}). We take as $Y$ the del Pezzo surface of degree 6, that is, the blow up of $\pp^2$ at three non collinear points; we denote by $|F_i|$, $i=1,2,3$ the three pencils of rational curves of $Y$ induced by the pencils of lines through the three blown up points. Recall that $K_Y\equiv-(F_1+F_2+F_3)$.  We take $G=\Z_3^2$ and consider  the $G$-cover $f\colon X\to Y$ given by the following choice of  building data:
\begin{itemize}
 \item[--]  $D_{10}\in|3F_1|$, $D_{01}\in|3F_2|$, $D_{22}\in|3F_3|$, and $D_v=0$ for the remaining $v\in G^*$;
 \item[--] $L_{10}\equiv F_1+2F_3$, $L_{01}\equiv F_2+2F_3$. 
 \end{itemize}
In addition we assume that $D=D_{10}+D_{01}+D_{22}$ is a simple normal crossings divisor. 
It is easy to check (cf. Remarks \ref{rem: reduced} and \ref{rem: Zp-covers}) that $L_{11}\equiv F_1+F_2+F_3$, $L_{22}\equiv 2(F_1+F_2+F_3)$ and for the remaining  $\chi$ one has $L_{\chi}\equiv F_i+2F_j$ for some $i\ne j\in\{1,2,3\}$.
The surface $X$ is smooth by Proposition \ref{prop: smoothcover}; by \eqref{eq: adjunction} we have  $3K_X\equiv f^*(3(F_1+F_2+F_3))$ and therefore $X$ is minimal of general type with $K^2_X=54$. Using \eqref{eq: hi} one computes $q(X)=0$ and $p_g(X)=8$.

Denote by $\Gamma<G$ the cyclic subgroup generated by $(1,2)$; one has $h^0(K_Y+L_{\chi})=0$ for every $\chi\notin \Gamma^{\perp}$, so the canonical map of $X$ is the composition of the  quotient map $X\to Z:=X/\Gamma$ with the canonical map of a smooth model of $Z$ (see  Remark \ref{rem: omega}). 
The surface $Z$ is a simple $\Z_3$-cover of $Y$ with branch locus  $D=D_{10}+D_{01}+D_{22}$ with canonical singularities (it has $A_2$ points over the singular points of $D$). We claim that the canonical map $\fie_Z$  of $Z$ is birational. Indeed, by \eqref{eq: simple-adj} the canonical bundle $K_Z$ is the pullback of $K_Y+2L_{11}\equiv F_1+F_2+F_3$, so the covering map $Z\to Y$ is composed with $\fie_Z$. Since $p_g(Z)=8>7=h^0(F_1+F_2+F_3)$, we conclude that $\fie_Z$ is not equal to the covering map and therefore it has degree 1.
\end{ex}

\begin{ex}[Case (B), $d=3$]  \label{ex3 rita} 
This example is taken from \cite{supcan}.
We take $G=\Z_3^3$ and $Y=\pp^2$; we denote by $\chi_1,\chi_2,\chi_3$ the canonical basis   of the space of characters and by
 $h$ the class of a line in $Y$.  We consider  the $G$-cover $f\colon X\to Y$ given by the following choice of reduced building data:
\begin{itemize}
 \item[--]  $D_v$ is a line if $\chi_1(v)=1, \chi_2(v)=0$ or  $\chi_1(v)=0, \chi_2(v)=1$, and $D_v=0$ otherwise;
 \item[--] $L_{100}\equiv L_{010}\equiv h$, $L_{001}\equiv 2h$;
\end{itemize}
In addition we assume that the lines $D_v$ are in general position. The surface $X$ is smooth by Proposition \ref{prop: smoothcover}; by \eqref{eq: adjunction} we have  $3K_X\equiv f^*(3h)$ and therefore $X$ is minimal of general type with $K^2_X=27$. 

Let $\Gamma$ be the subgroup of $G$ generated by $(0,0,1)$: it is not hard to check (cf. Remarks \ref{rem: reduced} and \ref{rem: Zp-covers}) that for all $\chi\notin \Gamma^{\perp}$ one has  one has $L_{\chi}\equiv 2h$ and therefore $h^0(K_Y+L_{\chi})=0$. Using \eqref{eq: hi} one computes $q(X)=0$ and $p_g(X)=5$.
By Remark \ref{rem: omega} the canonical map of $X$ is composed with the quotient map $X\to Z:=X/\Gamma$. The surface $Z$ is a $\Z_3^2$-cover with reduced building data:
\begin{itemize}
 \item[--]  $D_{10}=D_{100}+D_{101}+D_{102}$  and $D_{01}=D_{010}+D_{011}+D_{012}$;
 \item[--] $L_{10}\equiv L_{01}\equiv h$.
  \end{itemize}
 So $Z$ is the fibered product of two simple cyclic $\Z_3$-covers, each branched on the union of three lines;  it is easily seen to be isomorphic to the complete intersection of two cubics of $\pp^4$.  The singularities of $Z$ are $A_2$-points, occurring over the singular points of $D_{10}$ and $D_{01}$. 
So $Z$  is a canonically embedded surface and so the canonical map of $X$ has degree 3.

\end{ex}
\begin{ex}[Case (B), $d=5$]  \label{ex5} 
This example has been found  both in \cite{tan} and \cite{supcan}, independently. 

The construction is similar to the previous one. 
We take $G=\Z_5^2$ and $Y=\pp^2$; we denote by $\chi_1,\chi_2$ the canonical basis of the space of characters and by $h$ the class of a line in $Y$.  We consider  the $G$-cover $f\colon X\to Y$ given by the following choice of reduced building data:
\begin{itemize}
 \item[--]  $D_v$ is a line if $\chi_1(v)=1$ and $D_v=0$ otherwise;
 \item[--] $L_{10}\equiv h$, $L_{01}\equiv 2h$.
\end{itemize}
In addition we assume that the lines $D_v$ are in general position. The surface $X$ is smooth by Proposition \ref{prop: smoothcover}; by \eqref{eq: adjunction} we have  $5K_X \equiv f^*(5h)$ and therefore $X$ is minimal of general type with $K^2_X=25$. 

Let $\Gamma$ be the subgroup of $G$ generated by $(1,0)$: it is not hard to check (cf. Remarks \ref{rem: reduced} and \ref{rem: Zp-covers}) that for all $\chi\notin \Gamma^{\perp}$ one has  one has $L_{\chi}\equiv 2h$ and therefore $h^0(K_Y+L_{\chi})=0$. Using \eqref{eq: hi} one computes $q(X)=0$ and $p_g(X)=4$. By Remark \ref{rem: omega} the canonical map of $X$ is composed with the quotient map $X\to Z:=X/\Gamma$. The surface $Z$ is a simple  $\Z_5$-cover of the plane branched over the union $D$ of the 5 branch lines; 
its singularities are $A_4$-points, occurring over the singular points of  $D$. It is immediate to see $Z$ is isomorphic to a quintic in $\pp^3$, so it is a canonically embedded surface and the canonical map of $X$ has degree 5. 
 \end{ex}
\begin{ex}[Case (B), $d=2$, unbounded $p_g$]  \label{ex2 bin} 

This example is  one of the examples constructed by  Nguyen Bin in \cite{bin2}.

We take $G=\Z_2^3$ and $Y=\pp^1\times \pp^1$; we denote by $\chi_1,\chi_2,\chi_3$ the canonical basis of the space of characters,  by $F_1$ and $F_2$ the classes of the two rulings of $Y$.  We consider  the $G$-cover $f\colon X\to Y$ given by the following choice of reduced building data:
\begin{itemize}
 \item[--]  $D_{100}$, $D_{101}\in |2F_1+2F_2|$, $D_{110}\in |2mF_1|$,  $D_{111}\in |2nF_2|$ with $m,n \ge 2$, and $D_v=0$ otherwise;
 \item[--] $L_{100}\equiv (m+2)F_1+(n+2)F_2$, $L_{010}\equiv mF_1+nF_2$, $L_{001}\equiv F_1+(n+1)F_2$.
\end{itemize}
In addition we assume that the divisor $D=\sum_vD_v$ is a simple normal crossings divisor. The surface $X$ is smooth by Proposition \ref{prop: smoothcover}; by \eqref{eq: adjunction} we have  $2K_X\equiv f^*(2mF_1+2nF_2)$ and therefore $X$ is minimal of general type with $K^2_X=16mn$. 
Denote by $\Gamma$ the subgroup generated by $(001)$. Writing out \eqref{eq: p-power} for all the remaining  characters of $G$, one sees that for every $\chi$ one has: (1) $L_{\chi}\equiv aF_1+bF_2$ with $a,b>0$ for every $0\ne \chi\in G^*$;  (2) if $\chi\notin \Gamma^{\perp}$ then either $a$ or $b$ is equal to 1. Condition (1) gives $q(X)=0$ (cf.  \eqref{eq: hi}). Condition (2) implies that $h^0(K_Y+L_{\chi})=0$ for all $\chi\notin \Gamma^{\perp}$,  
so by Remark \ref{rem: omega} the canonical map of $X$ is composed with the quotient map $X\to Z:=X/\Gamma$. The surface $Z$ is a $\Z_2^2$-cover with  building data:
\begin{itemize}
 \item[--]  $D_{10}=D_{100}+D_{101}$  and $D_{11}=D_{110}+D_{111}$ and $D_{01}=0$;
 \item[--] $L_{10}\equiv (m+2)F_1+(n+2)F_2$, $L_{01}\equiv mF_1+nF_2$ and $L_{11}\equiv 2F_1+2F_2$
  \end{itemize}
The singularities of $Z$ are $A_1$ points occurring above the singular points of $D_{10}$ and $D_{11}$.
The cover $Z\to Y$ can be factored as a composition of two double covers $Z\to W:=Z/\Gamma \to Y$, where $\Gamma$ is the subgroup generated by $(11)$. The double cover $h\colon W\to Y$ is branched on $D_{10}\equiv  4(F_1+F_2)$, so $W$ is  a K3 surface with singularities of type $A_1$. The double cover $p\colon Z\to W$ is a flat double cover  given by the relation $2L\equiv B$, where $L:=h^*L_{01}$ and $B= h^*(D_{11})$.   We have $K_Z\equiv p^*(K_W+L)\equiv p^*L$, so $K^2_Z=8mn$ and $p_g(Z)=h^0(K_W)+h^0(K_W+L)=1+(2+2mn)=3+2mn$. 

The last step is  to prove that the canonical map of $Z$ is birational. One looks first at the system $|L|$ on $W$: the line bundle $L$ is the pull back of the  very ample line bundle  $mF_1+nF_2$ of $Y$ and in addition $h^0(L)=2mn+2>(m+1)(n+1)=h^0(mF_1+nF_2)$, so $|L|$ gives a birational map. A similar argument applied to $K_Z=p^*L$ shows that $|K_Z|$ is birational. 

The invariants of $X$ are: $K^2_X=16mn$, $p_g(X)=mn+3$ and $q(X)=0$.
 \end{ex}

\section{Remarks and open questions} \label{sec: questions}

The examples of the previous section show the great variety of possible behaviours of the canonical map. In this final section we try to give a  synthesis of the situation and formulate some questions. 
\subsection{Case (A) of Theorem \ref{thm: be-pg} ($p_g(\Sigma)=0$)} 

We have seen in  \S \ref{ssec: bounds} that if $p_g$ is large enough ($\ge 30$) the degree $d$ of the canonical map is at most 8. 

 Unbounded families of examples are known for all the even values $d\le 8$. 
As explained in \S \ref{sec: examples}, Example \ref{ex2A}, for $d=2$ there are examples with any  $p_g$  obtained by taking suitable double covers of surfaces with $p_g=0$.  

The product of two hyperelliptic curves gives rise to  examples with unbounded $p_g$ and $d=4$.  

Examples \ref{ex: GpairA},  \ref{ex6}  and \cite[Exemple 4.4]{be}  give   families with $d=6$ and  unbounded $p_g$.
 
As already mentioned in \S \ref{ssec: bounds},  \cite{be}  (see Example \ref{ex8 Be})   gives a family with $d=8$,  unbounded $p_g$ and $q=3$, whilst more recently several  such families with $q=0$ and $q=1$ have been  constructed in \cite{bin8}  as $\mathbb Z_2^3$-covers of  $\pp^2$ blown-up in one point. 

On the other hand,  for $d=3,5,7$ only sporadic examples are known  so far,  all having  relatively small $p_g$.  In \cite{triple}  surfaces with canonical map of degree 3, having a smooth canonical curve and whose image is a surface of minimal degree $n-1$ in $\pp^n$   are completely classified and explicitly constructed.  In particular surfaces with $K^2\leq 3p_g-5$ and canonical map of degree 3 necessarily have $p_g\leq 5$ (cf. \cite{zucconi3}).    So we are led to ask: 

\begin{qst} For $d=3,5,7$, are the invariants of surfaces  in  case (A) of Theorem \ref{thm: be-pg} bounded? 
\end{qst}

Another interesting situation  is $p_g=3$, since in this case the bound on the degree is very large, i.e. $d\le 36$ (cf. \S \ref{ssec: bounds}).  
\begin{qst} For $p_g=3$, are there examples for every  possible  $2\le d \le 36$? 
\end{qst}

For $d\leq 9$, $q=0$ the answer is affirmative (see, e.g., Example \ref{ex: pg}). As already mentioned in \S  \ref{ssec: image}, the limit values $36$ for $q=0$ and $27$ for $q=1$ given by  inequality  \eqref{ineA} were shown to be effective in \cite{Rito36}.    In recent times this question has been the subject of intense activity and,  without being exhaustive, we would like to mention  the example in \cite{Rito24}  with   $d=24, K^2=24,p_g=3, q=0$ and    the examples in  \cite{GPR}  with $d=32, K^2=32, p_g=3, q=0$ and $d=24, K^2=24, p_g=3,  q=1$.  However we believe that there are many gaps to be filled, for instance  many of  the cases of odd degree. 
 \smallskip 

 We note that if $p_g=q=3$ the possible degrees allowed by  inequality \eqref{ineA} vary between 2 and 9  but by the classification of surfaces with $p_g=q=3$ (see \cite{ha}, \cite{pi}) only $d=4$, $d=6$ and $d=8$  occur (see \cite{ccm}).  
 
  \bigskip

Many other such questions  arise. For instance:

\begin{qst} What pairs $(p_g, d)$  with $d\geq 9$ actually occur? \end{qst} 

The bounds on $p_g$ given by inequality \eqref{ineA} for $d\geq 10$ and the one for $d=9$ by \cite{xiaodeg}  (see \S \ref{ssec: image}) may not be the best possible and it would be interesting to know how sharp they are. 

  As an example, for $d=16$ we know,  by inequality \eqref{ineA}, that $p_g\leq 5$ and if $p_g=5$ then $q=0$. However although  examples with $d=16$ are known for $p_g=3, q=0$ (see Example \ref{ex: 16}),  $p_g=3, q=2$ (\cite{Rito16}) and $p_g=4,q=1$ (\cite{bin16}) the question of the existence of a surface with $d=16$ and $p_g=5$ is still open.

\subsection{Case (B) of Theorem \ref{thm: be-pg} ($p_g(\Sigma)=p_g(X)$)} 

The case when the canonical map has degree $d>1$ and its image is a canonically embedded surface is perhaps the most intriguing one. We have seen in \S \ref{ssec: bounds} that for $p_g\geq 13$  one has $d\le 3$ and that    there are unbounded sequences of such surfaces with $d=2$ (Examples \ref{ex: Gpair1}, \ref{ex: Gpair2}, \ref{ex: Gpair3}, \ref{ex2 bin}).

 On the other hand,  for $d=3$ the  examples we are aware of (Examples \ref{ex3 tan} and \ref{ex3 rita}, \cite{tan}, \cite{Rito3}, \cite{bin3}) satisfy $p_g\leq 8$. So the first question one is led to ask is: 

\begin{qst} For $d=3$, are the invariants of surfaces  in  case (B) of Theorem \ref{thm: be-pg} unbounded? 
\end{qst}
Note that possible   examples have    irregularity $q\le 3$ (see \S \ref{ssec: bounds}).  This shows 
that one cannot hope to construct a sequence of examples using a generating pair as in \S \ref{ssec: Gpair}:  by Corollary  \ref{cor: Gpair}  the general curve of $|L|$ would have genus  $g\ge 3$ and so by Lemma \ref{lem: invariants-Gpair} the examples would have irregularity $\bar g-g\ge 2(g-1)\ge 4$, a contradiction. 
In addition, for $d=3$ the only possible accumulation point for the slope $K^2/\chi$  is 9. 

These remarks suggest the following questions: 
\begin{qst} For $d=2$ is there an upper bound for the irregularity of a   surface  in  case (B) of Theorem \ref{thm: be-pg}? 
\end{qst}
\begin{qst}
For $d=2$ what  are  the accumulation points of the slopes $K^2/\chi$ of surfaces  in  case (B) of Theorem \ref{thm: be-pg}?
\end{qst}
We know by inequalities \eqref{acB2} and \eqref{acB3} of  \S \ref{ssec: bounds} that these accumulation points must lie in the interval $[6,9]$.  The sequences of surfaces constructed in Examples \ref{ex: Gpair1}, \ref{ex: Gpair2}, \ref{ex: Gpair3} have slopes tending to $6, \frac{32}{5},\frac{48}{7}$, respectively; the examples in \cite{cirifra}  are obtained from these by taking the quotient by a suitable $\Z/3$ action and their slopes have the same accumulation points.   Example \ref{ex2 bin} and the other examples of \cite{bin2}   have slopes tending to $8$.
\bigskip

For small values of $p_g$, in principle  it is possible to have  $3<d\le 9$ in case (B); more precisely we have the following bounds (see \ref{ssec: bounds}): 
\begin{itemize}
\item[--] if  $d= 4$, then $p_g\le 9$;
\item[--] if  $d= 5$, then $p_g\le 7$;
\item[--] if $d=6$, then $p_g\le 5$;
\item[--] if $d=7, 8 $ or $9$, then $p_g=4$.
\end{itemize} 

However except for 
Example \ref{ex5}  that  satisfies $d=5$ and $p_g=4$  we do not know other examples with $d>3$; the obvious question is:
\begin{qst}
For what pairs $(d,p_g)$, with $d>3$,  are there examples of surfaces in  case (B) of Theorem \ref{thm: be-pg}?
\end{qst}

\bigskip

\bigskip

\begin{minipage}{13cm}
\parbox[t]{6.5cm}{Margarida Mendes Lopes\\
Departamento de  Matem\'atica\\
Instituto Superior T\'ecnico\\
Universidade de Lisboa\\
Av.~Rovisco Pais\\
1049-001 Lisboa, PORTUGAL\\
mmendeslopes@tecnico.ulisboa.pt
  } \hfill
\parbox[t]{5.5cm}{Rita Pardini\\
Dipartimento di Matematica\\
Universit\`a di Pisa\\
Largo B. Pontecorvo, 5\\
56127 Pisa, Italy\\
rita.pardini@unipi.it}
\end{minipage}


\begin{thebibliography}{ABCDE} 
 \bibitem[Be79]{be} A. Beauville, {\em L'application canonique pour les surfaces de type g\'en\'eral}. Inv. Math. {\bf 55} (1979), 121--140.
\bibitem[BL94]{BL} Ch.~Birkenhake, H.~Lange, {\em A family of abelian surfaces and curves of genus four}, Manuscr. math., 85 (1994), 393--407.
\bibitem[Ca81]{babbage} F. Catanese, {\em Babbage's conjecture,  contact of surfaces, symmetric determinantal varieties and applications},  Invent. Math.,
{\bf 63} (1981), 433--465.
\bibitem[Ca87]{Cat} F.~Catanese, {\em Canonical rings and ``special" surfaces of general type},  Proceedings of Symposia in Pure Mathematics, vol. {\bf 46}
(1987), 175--194.
\bibitem[CCM98]{ccm} F.~Catanese, C.~Ciliberto, M. Mendes Lopes {\em On the classification of irregular surfaces  of general type with non birational bicanonical map}, Trans. Amer. Math. Soc. {\bf 350} (1) (1998), 275--308.
\bibitem[CPT00]{cirifra} C. Ciliberto, R. Pardini, F. Tovena, {\em Prym
varieties and the canonical map of surfaces of general type},  Annali della Scuola Normale
Superiore di Pisa. Classe di Scienze, XXIX (2000), 905--938.
 \bibitem[CPT03]{cirifra2} C. Ciliberto, R. Pardini, F. Tovena, {\em Regular canonical covers}, Math. Nachrichten 251 (2003), 19-27. 
\bibitem[De82]{debarre} O. Debarre, {\em In\'egalit\'es num\'eriques pour les surfaces de type g\'en\'eral}, with an appendix by A. Beauville, Bull. Soc. Math. France {\bf 110} 3 (1982),  319--346.
 \bibitem[GPR]{GPR} C.~Gleissner, R.~Pignatelli, C.~Rito,
 {\em New surfaces with canonical map of high degree},
Commun. Anal. Geom, to appear, arXiv:1807.11854 [math.AG]
\bibitem[HP]{ha} C. Hacon, R. Pardini, {\em Surfaces with $p_g=q=3$},  Trans. Amer. Math. Soc. {\bf 354} (7) (2002), 2631-2638.
\bibitem[Ho]{ho1} E. Horikawa, {\em Algebraic surfaces of general type with small $c_1^2$, I}, Annals of Mathematics, {\bf 104} (1976), 357--387.
\bibitem[MP98a]{nagoya} M.~Mendes Lopes, R.~Pardini, {\em Irregular canonical double surfaces}, Nagoya J. of Math., 152 (1998), 203-230.
\bibitem[MP98b] {triple} M.~Mendes Lopes, R.~Pardini, {\em Triple canonical  surfaces of minimal degree},
International Journal of Mathematics,  {\bf 152} (1998), 203-230.
\bibitem[Ng19a]{bin16} Bin~Nguyen,  {\em A new example of an algebraic surface with canonical map of degree 16}, Archiv der Mathematik, volume 113, 385--390(2019)
\bibitem[Ng19b]{bin3} Bin~Nguyen,  {\em New examples of canonical covers of degree 3}, Math. Nach., to appear,  arXiv:1910.03526 [math.AG]
\bibitem[Ng20]{bin8} Bin~Nguyen,  {\em Some unlimited families of minimal surfaces of
general type with the canonical map of degree 8},  manuscripta math. {\bf 163}, 13--25 (2020). 
\bibitem[Ng21]{bin2} Bin~Nguyen, {\em Some infinite sequences of canonical covers of degree 2}, Advances in Geometry,  {\bf 21}  (1) (2021), 143--148.
\bibitem[Pa91a]{ritaabel} R. Pardini, {\em Abelian covers of algebraic varieties},   J. reine angew. Math.  {\bf 417} (1991), 191--213.
\bibitem[Pa91b] {supcan} R. Pardini, {\em Canonical images of surfaces}, J.
reine angew. Math., {\bf 417} (1991), 215--219.
\bibitem[Pe77]{Pe} U.~Persson,  {\em Double coverings and surfaces of general type}, Algebraic
geometry (Proc. Sympos., Univ. Troms\o, Troms\o, 1977), vol. 687 of
Lecture Notes in Math., Springer, Berlin (1978), 168--195.
\bibitem[Pi02]{pi} G.P.~Pirola, {\em Surfaces with $p_g=q=3$}, Manuscripta Math., {\bf 108} (2002), 163--170. 
\bibitem[Ri16]{Rito3} C.~Rito, {\em Cuspidal quintics and surfaces with $p_g=0$, $ K^2=3$ and 5-torsion},
LMS J. Comput. Math. {\bf 19} (1)  (2016), 42--53.
\bibitem[Ri17a]{Rito16} C.~Rito, {\em A surface with q=2 and canonical map of degree 16}, Michigan Math. J., {\bf 66}  (1),  (2017), 99--105.
\bibitem[Ri17b] {Rito24}C.~Rito,
 {\em A surface with canonical map of degree 24},
Int. J. Math. {\bf 284}  (6),  (2017), 75--84.
\bibitem[Ri19]{Rito36} C.~Rito, {\em Surfaces with canonical map of maximum degree}, Journal of Algebraic Geometry, to appear,  	arXiv:1903.03017 [math.AG]
\bibitem[Ta92]{tan}  S.~Tan {\em Surfaces whose canonical maps are of odd degrees}, Math. Ann. {\bf 292} (1992), 13---29.
\bibitem[Xi85]{xiaopencil} G. Xiao, {\em L?irr\'egularit\'e des surfaces de type g\'en\'eral dont le syst\`eme canonique est compos\'e d?un pinceau},  Compositio Math., {\bf 56} (2):251--257, 1985.
\bibitem[Xi86]{xiaodeg} G. Xiao, {\em Algebraic surfaces with high canonical degree}, Math. Ann. {\bf 274} (1986), 473-483.
\bibitem[Zu97]{zucconi3} F. Zucconi,  {\em Surfaces with  canonical map of degree 3 and $K^2=3p_g-5$}, Osaka J. Math.,  {\bf 34} (2), (1997), 411--428.
\end{thebibliography}
     \end{document}